\documentclass[12pt]{article}
\usepackage{graphicx}
\usepackage{amsfonts}
\usepackage{amssymb}
\usepackage{amsmath}
\usepackage{enumerate}

\providecommand{\U}[1]{\protect\rule{.1in}{.1in}}
\renewcommand{\baselinestretch}{1.5}

\newtheorem{theorem}{Theorem}[section]

\newtheorem{assumption}{Assumption}[section]

\newtheorem{definition}{Definition}[section]

\newtheorem{lemma}{Lemma}[section]

\newtheorem{remark}{Remark}[section]

\newenvironment{proof}[1][Proof]{\noindent\textit{#1.} }{\ \hfill\rule{0.3em}{0.5em}\\}
\def\E{{\rm E}}
\begin{document}

\title{A General Framework of Multi-Armed Bandit Processes\\by Arm Switch Restrictions\thanks{%
This research was partially supported by: (1) Natural Science Foundation of China (Nos. 71531003, 71432004), Research Grants Council of Hong Kong (No. T32-102/14), and the Leading Talent Program of Guangdong Province (No. 2016LJ06D703) to Xiaoqiang Cai and (2) Natural Science Foundation of China (Nos. 71371074, 71771089) and  the 111 Project (grant No. B14019) to Xianyi Wu.}}

\author{Wenqing Bao$^{1}$, Xiaoqiang Cai$^2$, Xianyi Wu$^{3}$ \and  \and $^{1}$ Zhejiang Normal University,
Jinhua Zhejiang, P.R. China \and $^2$
The Chinese University of Hong Kong (Shenzhen) and \and
The Shenzhen Research Institute of Big Data, Guangdong , P.R. China \and $^{3}$ East China Normal
University, Shanghai, P.R. China}
\date{}
%\tableofcontents

\maketitle

\begin{abstract}
%This paper is devoted to  optimal allocation of  multi-armed bandit processes  evolving in continuous time, in which each  arm is associated with a random time set and a switch from an arm to another is allowed only when its processing time belongs to that random time set.
This paper proposes a general framework of multi-armed bandit (MAB) processes by introducing a type of restrictions on the switches among arms evolving in continuous time. 
 The Gittins index process is constructed for any single arm subject to the restrictions on switches  and then the optimality of the corresponding Gittins index rule is established.  The Gittins indices defined in this paper are consistent with the ones for  MAB processes in continuous time, integer time, semi-Markovian setting as well as general discrete time setting, so that the new theory covers the classical models as special cases and also applies to many other situations that have not yet been touched in the literature. While the proof of the optimality of Gittins index policies benefits from  ideas in the existing theory of MAB processes in continuous time, new techniques are introduced  which drastically simplify the proof.

%\bigskip

 %\noindent\textit{Review area}: Stochastic Models

%\noindent\textbf{MSC2010 Classifications}: 90B36, 60G40, 90C40.
\end{abstract}

\section{Introduction}

 Multi-armed bandit  processes (MAB in short) model the resource allocation problem with uncertainties where a decision maker attempts  to optimize  his decisions based on the existing knowledge, so as to maximize his expected  total reward over time (Gittins et al., 2011). It has extensive applications in clinical trials, design of experiments,  manufacturing systems, economics, queuing and
communication  networks,   control  theory,  search  theory,   scheduling, machine learning/reinforcement learning etc.

In this paper we are concerned with a general multi-armed bandit problem with restricted
random stopping time sets, which can be roughly described as follows: There is a multi-armed bandit process  consisting of a set of $d$ statistically
independent arms evolving in continuous time among which a
resource  (time, effort) has to be allocated. Every arm is associated with a restricted stopping time set, in the sense that the arm must be engaged exclusively if its operation time does not belong to the stopping time set. The allocation respects
the restrictions and any engaged arm accrues rewards
that are represented as a general stochastic process.
The objective is to maximize the total expected discounted reward over an
infinite time horizon.

The early versions of \textit{discrete-time} MAB processes in Markovian and semi-Markovian fashions
have been well understood
due to the pioneer work of Gittins and Jones
 (1972) and subsequently the seminal contributions  of Gittins  (1979, 1989) and Whittle
 (1980, 1982). The significance of Gittins' contribution is the drastic
dimension reduction: Instead of solving the optimal problems of the Markov
 (or semi-Markov) decision models formed by all arms, one only needs to
compute an index function  of the states based merely on the information
delivered by each arm itself and then picks an arm with the highest index to operate. That index function, known generally as Gittins Indices today, was defined by Gittins as the maximum reward rate over all arm-specified stopping times, Whittle (1980) provided a mathematically
elegant proof by showing that Gittins index policies solve the optimality equations
of the corresponding dynamic programming modeling the multi-bandit
processes. For general reward processes in integer time (without Markovian assumption), Varaiya et. al.  (1984) defined an
optimal policy in abstract terms by reducing every $d$-armed problem
to $d$ independent stopping problems of the type solved by Snell  (1952).
Mandelbaum  (1986) proposed a technically convenient framework by formulating  a control problem with time parameters in a multidimensional,
partially ordered set. EL Karoui and Karatzas  (1993) presented a
mathematically rigorous proof of Gittins index policies for arbitrary stochastic processes evolving in
integer times by combining the formulation of Mandelbaum  (1986) with ideas
from Whittle  (1980). The most general treatments for discrete time setting can be found in Cai, et al. (2014, Section 6.1) and Cowan and Katehakis  (2015) by dropping the Markovian property from the semi-Markovian model so that switches from one arm to another can only take place at certain time points and the intervals between any pair of consecutive points are random quantities.
%This extends the semi-Markovian setting to a general one with ``states'', where the Markovian assumption is replaced by filtrations standing for the information histories. 
One key feature in discrete time setting is that the switches from any arm can only occur in countably many time instants, even though the arms can evolve continuously over the time horizon. This type of problems are referred to as general discrete time setting . Some aspects of the theory in the discrete time version and applications  in
searching, job scheduling, etc.,  can also be found in the comprehensive monograph by Gittins
et al.  (2011) .

The parallel theory for MAB processes in \textit{continuous time} was not developed until a later time, due to mainly the technical intricacy in mathematics, where the term ``continuous time'' emphasizes not only that rewards can be continuously collected but, most significantly in mathematics, that  switches from one arm to another are allowed to be made at arbitrary time points in $ (0, \infty)$ also, such that the time set for an arm from which switches can be made is the whole positive axis, i.e., essentially uncountable, sharply in contrast to the discrete time version in which the switches are essentially countable.
It is consensus  that continuous time stochastic processes are far more difficult to attack than their discrete time versions, due to the difficulties in dealing with the measurability of the quantities involved. 
As to the continuous time version of the problem in a
Markovian case, relevant results were first obtained by Karatzas  (1984) and
Eplett (1986). By insightfully formulating the model as a stochastic control problem
for certain multi-parameter processes, Mandelbaum  (1987)
extended the problem to a general dynamic setting. Based on Mandelbaum's
formulation, EL Karoui and Karatzas  (1994) derived general results by combining
martingale-based methodologies with the retirement option designed by Whittle  (1980) for his elegent proof of the optimality of Gittins index policies in discrete time. These results were further revisited by Kaspi
and Mandelbaum  (1998) with a relatively short and rigorous proof by means of  excursion
theory.

To sum up, studies on MAB processes have treated only the two regular ends: the discrete time version (including the  general discrete time setting)
in which switches from any arm to another are at most countably infinite, and the continuous time version in which the controller can switch from one arm to another in any time point in the positive time horizon, with technically different methods.

Clearly, in between the two regular ends, there exist many real-life situations that could not be put in the framework formed by solely either of the two versions, especially when there are technical restrictions on the switch times of the arms.
As an example, consider a simple job scheduling scenario subject to machine breakdowns  (see, e.g., Cai et al, 2014), in which a single unreliable machine is to process a set of jobs and, in serving the jobs, the machine may be subject to breakdowns from time to time, caused by, for instance, damage of components of the machine or power supply. When the machine
is workable, a job can be processed and the processing can be preempted so as to switch
the machine to any one of the unfinished jobs. Once the machine is
broken down, it must be continuously repaired until it can resume its operation
again. In this scenario, the stopping times for the machine to be switched from one job to another are restricted to the time interval
in which the machine is in good condition. By associating the repairing duration of the machine to the job being processed, this problem can be modeled by a multi-armed bandit process. This bandit process, however, cannot  be put in any of the frameworks of discrete time and  continuous time bandit processes, owing to  two significant features: First, for any job, the set of its potential switching times are essentially continuum in the interval in which the machine is workable so that the framework cannot be the discrete time version. Second, in the time intervals of machine reparation,  a switch from the job is prohibited so that the framework cannot be the continuous time version. 
As another example that the classical MAB models cannot accommodate, consider a second job scheduling problem in which some of the jobs can be preempted at any time points, whereas the other jobs consist of a number of nonpreemptable components so that, once a job is selected to process, it could not be preempted until  the completion of a component. This problem can be translated to such an MAB formula that  some arms evolve in continuous time setting and the others respect to a discrete time mechanism.  
Furthermore, one can even image such situations where  jobs consist of possibly preemptable and  nonpreemptable components, so that, being represented as MAB models, the arms can be  in continuous time,  discrete time version or in a mixture mode in which the switch times  contain both continuum and discrete time parts.
Clearly, the existing optimality theory of MAB processes is not applicable to these situations.

This paper is dedicated to  propose
a new MAB process model so as to accommodate these situations. This is accomplished  by introducing  a type of restrictions on switch times, or equivalently the arm-specified stopping
times as what discussed recently in Bao et al (2017) for restricted optimal stopping problems. 
Firstly, it turns out that this new model also unifies the existing versions of MAB processes.
Specifically, for the  sole discrete time version, the switching times of every arm are only the integer times, for the  general discrete time version, the switching times are clearly the end points of the intervals  during which no switch is  allowed and the purely continuous time setting corresponds simply to the case of no restriction (see Section \ref{ModelSpecificationSection} for details).
Moreover, an obvious merit of this new framework is, by introducing different restrictions on different arms, it can give the optimal solution to irregular cases in which some of the arms follow continuous time,  some others follow discrete time  and still others even respect more complicated mixtures;
see the examples above. Such important types of MAB processes have not yet been touched in the existing literature.

To successfully tackle this problem, we will combine the martingale techniques as employed by EL Karoui and Karatzas  (1994) with the excursion method similar to that used by Kaspi and Mandelbaum
 (1998), but now under the new framework of  general $d$-armed bandit processes with each arm attached with a restricted stopping time set.
%It is worth noting that,
% %analogy to the situation of MAB processes,  except for
% in addition to practical applications  involving restricted stopping time sets in optimal stopping problems, a theoretical advantage of restricted sets of stopping times
% % for optimally stopping gain processes
% is  a unified approach that can deal with  discrete time, continuous time, and   general discrete time, which have so far been separately treated in the literature as we discuss above.

The main contribution of this paper consists of the following:
\begin{enumerate}
%\item The framework we develop unifies  all the MAB process models in the literature.
%\item It provides a new theory that can be applied to situations that could not be solved by the existing theory of MAB processes.
%\item The new model we formulate is substantially different from the one separately discussed by  Cai, Wu and Zhou (2014, Section 6.1) and  Cowan and Katehakis  (2015).
%In our model, the set of switching points (referred to as allowable time set below) for an arm can be random and incountable, whereas in Cai, Wu and Zhou (2014) and Cowan and Katehakis  (2015), the set of switching points for %an arm can only be random and countable. The continuum of the switching points is the main obstacle in theoretical analysis because it may cause the curse of measurability in mathematics, whereas, in the case of countably many %switch times, there is no question of measurability.

\item[(1)] We develop a general and new framework of MAB processes, suggest correspondingly a general definition of Gittins indices and demonstrate their optimality in arm allocation under switch time restrictions. This framework generalizes and unifies the models, methodologies and theory for  all versions of MAB processes and can apply to more other situations.
%which have so far been treated separately in the literature. It includes new problems that have not been addressed
%in the existing literature  but occur practically.

%\item[(2)] By extending the optimal stopping theory to allow
%restricted stopping times, we derive the optimal allocation strategies
%for general multi-armed bandit processes with each arm
%attached  with a restricted stopping time set.

\item[(2)] While the proof follows the ideas partly from EL Karoui and Karatzas  (1994) and  partly from Kaspi and Mandelbaum  (1998), new techniques (e.g., the discounted gain process \eqref{gain process} and Lemma \ref{lemma_pf_Th_1}) are introduced such that the proof is drastically shorter than the ones  for the unrestricted MAB processes in continuous time.

\end{enumerate}

The reminder of the paper is organized as follows.
%\begin{enumerate}
Section \ref{ModelSpecificationSection}  formulates the restricted MAB processes with each arm associated with a restriction on stopping times.  
 After a concise review of the theory of optimal stopping times
 with restrictions in Section \ref{rst_e}  so as to prepare some necessary theoretical foundation, Section \ref{GittinsIndexProcessSection} associates each arm with a Gittins index process defined under the restrictions on stopping times, which unifies and extends the classical definitions for discrete time, continuous time and semi-Markovian setting. The properties of the Gittins index process are also addressed there.   Section \ref{rsmb_proof_G_P} is dedicated to demonstrate the optimality of Gittins index policies. The paper is concluded in Section \ref{Conclusions} with a few remarks.
%\end{enumerate}

\bigskip

% Section \ref{rsmb_single-arm} and \ref{rsmb_proof_G_P} give a
%proof for the result, Section \ref{rsmb_proof_G_P} also presents a
%compositive two-armed bandit.
\section{Model Specification}\label{ModelSpecificationSection}
\setcounter{equation}{0}

The  MAB processes for which the switches among arms are subject to restrictions are  referred to as ``restricted multi-armed Bandit processes" (RMAB processes). 
 
%%%%%%%%%%%%%%%%%%%%%%%%%%%%%%%%%%%%%%%%%%%%%%%%%%%%%%

In this paper, a  RMAB process refers to a stochastic control process governed by the following
mechanism.
The primitives are $d$ stochastic processes $ (X^{k}, {\cal F}  %
^{k}), k=1, 2, \ldots, d$, evolving on $\mathbb{R}_{+}=[0, +\infty)$, all of
which are defined on a common probability space $ (\Omega, {\cal F} , P)$
to represent $d$ arms, meeting the following formulation:

\begin{enumerate}
\item {\textbf{Filtrations}}. For every $k\in\{1, 2, \dots, d\}$, ${\cal F}  %
^{k}=\{{\cal F}  _{t}^{k}, t\in\mathbb{R}_{+}\}$ is a
quasi-left-continuous filtration satisfying \textit{the usual conditions}
and ${\cal F}  _{0}=\{\emptyset, \Omega\}, $ mod $P$.
%,  to indicate the information accumulated due to arm $k$ when it has
%been operated for $t$ units of time.
%Conventionally write ${\cal F}  _{\infty}^{k}=\bigvee_{t=0}^\infty{\cal F}_t$.
The collection $\{{\cal F}  ^{1}, {\cal F}  ^{2}, \ldots, {\cal F}  %
^{d}\}$ of filtrations are assumed to be mutually independent.

\item {\textbf{Rewards}}. For every $k\in\{1, 2, \dots, d\}$,  $X_{t}^{k}\geq0$, the instant reward rate obtained at the moment when arm $k$ has just been
pulled for $t$ units of time,  is assumed to be ${\cal F}  ^{k}$-progressive and, with no loss of generality,  satisfies $
\E \left[ \int_{0}^{\infty}e^{-\beta t} X_{t}^{k}dt \right]  < \infty
$.

\item {\textbf{Restrictions}}. Let ${\cal M}^k$ be an  $%
{\cal F}  ^k$-adapted random time set,  referred to as the \textit{feasible
time set} of arm $k$, satisfying $0, \infty\in{\cal M}^k$ and ${\cal M}%
^k (\omega)=\{t: (t, \omega)\in{\cal M}^k\}$ is closed for every $\omega\in
\Omega$. For an ${\cal F}  ^k$-stopping time $\tau$, also write $\tau\in%
{\cal M}^k$ if $ (\tau, \omega)\in {\cal M}^k$ almost surely;  the symbol ${\cal M}^k$ refers to both a random set and the set of  stopping times $\tau$ with $ (\tau, \omega)\in {\cal M}^k$ a.s.. Here ${\cal M}^k$ may vary over $k$, subject to different
requirements.
%Let ${\cal S}=\prod_{k=1}^d {\cal M}^k$ stand for the
%$d$-fold Cartesian product of ${\cal M}^k, k=1, \cdots, d$.
\item {\textbf{Policies under restrictions}}. An allocation policy $T$ is characterized by a $d$-dimensional stochastic process %
$
T:=\{T (t):t\in\mathbb{R}_{+}\}=\{ (T^{1} (t), T^{2} (t), \ldots, T^{d} (t)):t\in%
\mathbb{R}_{+}\},
$
where $T^{k} (t)$ is the total amount of time that $T$ allocates to arm $k$
during the first $t$ units of calendar time, satisfying the following technical requirements: \vspace{-2mm}

\begin{itemize}
\item[ (1)] $T (t)$ is component-wise nondecreasing in $t\geq 0$ with $T (0)=%
\mathbf{0}$.

\item[ (2)] $T^{1} (t)+T^{2} (t)+\cdots+T^{d} (t)= t$ for every $t\geq 0$.

\item[ (3)] For any nonnegative vector $s= (s_1, s_2, \dots, s_d)\in \mathbb{R}%
^n_+ $, $\{T (t)\leq s\}\in {\cal F}  ^{1}_{s_1}\vee\cdots\vee{\cal F}  %
^{d}_{s_d}$.

\item[ (4)] $\frac{d^+T^{k} (t)}{dt}=1$ if $ (T^k (t), \omega)\in {\cal M}^{k}_c:=\mathbb{R}_+\times\Omega-{\cal M}^{k}$,
where   ${\frac{d^+}{dt}}$ indicates the right
derivative.
\end{itemize}
\item{\bf Objective.} With any policy $T$, {the total reward of the
bandit in calendar time interval $[t, t+dt]$ is} $%
\sum_{k=1}^dX^k_{T^k (t)}dT^k (t)$, so that
the total expected present value of this $d$-armed bandit system is
\begin{equation}
v (T)=\sum_{k=1}^d\E \left[\int_0^{\infty}e^{-\beta t} X^k_{T^k (t)}dT^k (t)\right],
\label{rsmb_eq_1}
\end{equation}where $\beta>0$ indicates the interest rate.
 The objective
%of the bandit problem
 is to find a policy $\hat{T}$  such that
$
v (\hat{T})=\max_{T}v (T)  \label{rsmb_eq_2}
$,
where the maximization is taken over all the policies characterized above.
\end{enumerate}

The following remarks give more details on the formulation of RMAB processes.

\begin{enumerate}[(a)]

 \item For the reward processes, the requirement $
\E \left[ \int_{0}^{\infty}e^{-\beta t} X_{t}^{k}dt \right]  < \infty,
k=1, 2, \dots, d$  makes the problem nontrivial,
 because, supposing it does not hold for some $k$,
 % i.e., $\E \int _{0}^{\infty}e^{-\beta t}X_t^{k}dt =\infty$,
then one can optimally obtain an
infinite expected reward by operating arm $k$ all the time.

 \item While, from a practical point of view, policies satisfying $T^{1} (t)+T^{2} (t)+\cdots+T^{d} (t)\leq t$ for every $t\geq 0$ allow for machine idle and are also practically feasible and can contain more policies than  those defined by condition (2) in the ``Policies under restrictions'' which does not allow
for machine idle. Nevertheless, by introducing a dummy
arm with constantly zero reward rate, constant filtration and the trivial
feasible random time set $[0, \infty)$, the setting in condition (2) can
model this more realistic situation.

 \item Conditions (1) -- (3) in  ``Policies under restrictions'' are similar to those in  Kaspi and Mandelbaum
 (1998), whereas condition (4) that is new captures the feature of restricted
policies that the machine can operate  arm $k$ at a rate strictly less than $1$ only when its
operation time is in ${\cal M}^k$; in other words, if $T^k (t)\in
{\cal M}^k_c$, then at time $t$, the machine can only
be occupied by  arm $k$ exclusively. 

 \item Clearly, the setting we have just formulated subsumes classical versions in discrete time, continuous time and  general discrete time setting, as discussed below:
 \begin{enumerate}[i)]
 \item Because ${\cal M}^k= (\mathbb{N}%
\cup\{\infty\})\times\Omega$ indicates that arm $k$ can be switched at only integer times, an integer time MAB process corresponds to a RMAB process in which ${\cal M}^k= (\mathbb{N}%
\cup\{\infty\})\times\Omega$ for every $k=1,2,\dots, d$.
\item  In the case of a \textit{semi-Markov process}, let $G_t^k$ be the state of the process and denote by $\tau_n^k$, $n=0, 1, \dots$, the time instants at which $%
G^k_t$ makes transitions, with $\tau^k_0=0$. Arm $k$ 
can only be switched  only at the time instants $\tau_n^k, n=0, 1, \dots$, so that 
\begin{equation}\label{semiMarkovianMAB}
{\cal M}^k=\{ (\tau_n^k (\omega), \omega): n=0, 1,
\dots, \omega\in\Omega\}\cup\{ (\infty, \omega):\omega\in\Omega\}. 
\end{equation}
A semi-Markovian MAB corresponds to a RMAB process with every ${\cal M}^k, k=1,2,\dots, d$ of the form  in \eqref{semiMarkovianMAB}. 
\item We in this item show how the RMAB processes can be reduced to  general discrete time MAB processes. Let
  $\{s_n:n\geq 1\}$ be a sequence of  increasing ${\cal F}^k$-stopping times at which arm $k$ can be  can be stopped to switch to another arm, satisfying $\Pr(s_n\geq s_{n-1})=1$ for all $n=1,2,\dots$ and $\lim_{n\rightarrow\infty}s_n=\infty$ a.s.. 
Clearly, for this example, 
\begin{equation}\label{ general discrete time}
{\cal M}^k=\{(s_n(\omega),\omega):n=0,1,\dots,\omega\in\Omega\}\cup\{ (\infty, \omega):\omega\in\Omega\}. 
\end{equation}
Also, an  general discrete time MAB corresponds to a RMAB process with every ${\cal M}^k$ having the form in Equation \eqref{ general discrete time}. This model extends the semi-Markov model by dropping the Markovian property in the transition. Note that this model essentially covers MAB in discrete time, because the evolving of the process   in between $s_{n-1}$ and $s_n$ are irrelevant for the purpose of making decision on stopping  at those stopping times $s_k,k=1,2,\dots$. It was discussed in Cai et al (2014, Section 6.1) and Cowan and Katehakis  (2015) when they discussed their multi-armed bandit processes. Clearly, RMAB process clearly covers  general discrete time model as a special case, but not vice versa because, as just stated, RMAB process covers the continuous time version of MAB whereas that discrete time version of MAB does not. 
\item   If ${\cal M}^k=[0, \infty]\times \Omega $, arm $k$ is an arm in \textit{continuous time} in which one can stop at any time,%\textbf{\ In the sequel,
%denote the set of all stopping times in ${\cal M}_{2}$ by ${\cal T}$.}
 and for optimal stopping
problem in \textit{discrete time}.

 \end{enumerate}
 \item Moreover, the restrictions allow one to tackle many more situations.
 Here is a selection of some examples, for all of which but the first the existing theory for MAB processes cannot apply.
\begin{enumerate}[i)]
\item If the case ${\cal M}^k=\{0, +\infty\} $, then ${\cal M}^k_c=\mathbb{R}_+\times \Omega $, so that the arm $k$ will be operated exclusively forever once it is picked. Obviously, it corresponds a nonpreemptable arm.
\item If ${\cal M}^k= [0, \tau]\cup\big\{n: n \hbox{ is positive integer in between }[\tau,\{+\infty\}\big\}\cup [+\infty]$,
 where $\tau$ is an ${\cal F}^k$-stopping time, so that
${\cal M}^k_c= (\tau,\infty)$, then switches from arm $k$ are all time points no larger than $\tau$ and the integer time points larger than $\tau$.
\item Let $s^k_n, n=1,2,\dots, \infty$ be a sequence of ${\cal F}^k$-stopping times increasing in $n$ and ${\cal M}^k= \bigcup_{n=1}^\infty[s_{2n-1},s_{2n}] \bigcup\{0,\infty\}\times \Omega$. Then arm $k$ can only be switched from at its private random time intervals $[s_{2n-1},s_{2n}], n=1,2,\dots$ whereas in its private time intervals $(s_{2n-2},s_{2n-1}), n=1,2,\dots$, the occupation of machine by this arm is exclusive.
\item  One can treat MAB processes of multiple types of arms, where operation on some of the arms can be switched to other arms at any time (corresponding to a continuous time setting) but operation on some other arms can only be switched when the machine has been served for  integer amount of time (discrete time setting) or when the state of the arm is just transferred in the case of the semi-Markovian setting. Some arms can even be nonpreemptable.
 
\end{enumerate} \end{enumerate}
 
% Hence, the complete solution to RMAB processes may provide deeper insight into the structure of MAB processes.
\section{Gittins Indices for A Single-Arm Process}
\label{rsmb_single-arm}
\setcounter{equation}{0}
After the RMAB processes were formulated in the last section, we now associate each arm with an appropriately defined Gittins index process, which unifies and extends the classical definitions for discrete time, continuous time and  general discrete time setting. 
Because we consider  only a single arm so as to define the associated Gittins index process and demonstrate its desired
properties, for the
time being,  the arm identifier $k$ is suppressed for the time being for notation convenience. Hence we work only with  a single stochastic process  $G= (G_{t})_{t\in\mathbb{\bar{R}}_{+}}$ that is ${\cal F}  $-adapted
 on a filtered probability space $ (\Omega, {\cal F} , P)$,  equipped with a quasi-left-continuous
filtration ${\cal F}  = (%
{\cal F}  _{t})_{t\in\mathbb{\bar{R}}_{+}}$  satisfying the \textit{usual conditions} of right continuity and augmentation by the null sets of ${\cal F}_\infty$, where $\mathbb{\bar{R}}_{+}=[0, +\infty]$. 
 To $ (\Omega, {\cal F} , P)$ is associated with a
random set  ${\cal M}$ to represent the restricted feasibility on the stopping times,  as defined in Section \ref{ModelSpecificationSection}.

This section consists of two parts: Section \ref{rst_e} gives a concise review of restricted optimal stopping times with some material taken from Bao et al. (2017), which is put here for easy reference and  in Section  \ref{GittinsIndexProcessSection} we define the Gittins index process  induced over a single arm and gives its details.

\subsection{Optimal stopping times under the restrictions}
\label{rst_e} 
%This section address the properties of a single arm with restricted stopping times.

%Section \ref{frmlt} formulates the optimal stopping problem with restricted stopping times, Section \ref{sec_generality} discusses the merits of this new framework,  and Section \ref{slt} establishes  the solution  by Theorems \ref{rts_Th1} and \ref{Th:2} with a series of lemmas. 
%\subsection{Formulation\label{frmlt}}
%\subsubsection{Mathematical formulation}
%The constrained optimal stopping problem is formulated as follows:

%The introduction of a restricted stopping time set $\cal M$ unifies and generalizes the classical
% theory of stochastic processes in discrete time, continuous time,  and  general discrete time setting to the framework of continuous time stochastic processes with restricted stopping times, especially in the analysis  of optimal stopping problems. This framework contains the classical models as simple or trivial special cases, as illustrated in Examples below. 
 
 The optimal stopping time problem with restrictions, denoted by $ (\Omega, {\cal F} , P, {\cal M})$, is defined as the following: For an arbitrary stopping time $\nu\in [0, \infty]$ (unnecessarily in $\cal M$), find a optimal stopping times $\tau^*\in\cal M$ such that
\begin{equation}
Z_{\nu}:=\E [G_{\tau^* }|%
{\cal F}  _{\nu}]=\operatorname{esssup}_{\tau\in{\cal M}_{\nu}}\E [G_{\tau }|%
{\cal F}  _{\nu}], \label{rts_eq:3}
\end{equation}
where  esssup stands for the operation of essential supremum, ${\cal M}_\nu=\{\tau\geq\nu: \tau \in\cal M \}$ and  $G$ is assumed to satisfy the following assumptions:
\begin{assumption}
\label{rts_assumption_1}
\hfill
\begin{description}
\item (1). $G$ has almost surely right continuous paths.

\item (2). $\E \left[ \sup_{t\in\mathbb{\bar{R}}_{+}}|G_{t}|\Big|%
{\cal F}  _{0}\right]  < {\infty}$.

\item (3). $\E [G_{\infty}] \geq \displaystyle\limsup_{t\rightarrow%
\infty}\E [G_{t}]$.% is taken for granted to ensure the existence of an optimal stopping time.
%
%\item (4). $\infty\in{\cal M}$ and ${\cal M} (\omega)=\{t: (t, \omega)\in%
%{\cal M}\}$ is closed for every $\omega\in \Omega$.
\end{description}
\end{assumption}
By Bao et al. (2017), problem \eqref{rts_eq:3} is solved by the following two theorems that are cited here for later reference. 
The first theorem characterizes the optimal stopping times should they
exist.
\begin{theorem}
\label{rts_Th1} The following three statements are equivalent for any $
\tau_{\ast}\in{\cal M}_{\nu}$:

\noindent (a) $\tau_{\ast}$ is  optimal for problem (\ref{rts_eq:3}), i.e., $%
Z_{\nu}=\E [G_{\tau_{\ast}}|{\cal F}  _{\nu}]$;

\noindent (b)   The stochastic process $\{Z_{\tau_{\ast
}\wedge (\nu\vee t)}:t\in\mathbb{\bar{R}}^{+}\}$ is an ${\cal F}  _{\nu\vee
t}$-martingale and $Z_{\tau_{\ast}}=G_{\tau_{\ast}}\ a.s.$;

\noindent (c) $Z_{\tau_{\ast}}=G_{\tau_{\ast}}\ a.s.$ and $Z_{\nu}=\E %
[Z_{\tau_{\ast}}|{\cal F}  _{\nu}]$.
\end{theorem}

\noindent For any $\lambda\in (0, 1)$ and stopping time $%
\nu$, define $D_{\nu}^{\lambda}=\operatorname{essinf}\{\tau\in{\cal M%
}_{\nu}:\lambda Z_{\tau}\leq G_{\tau}\} $ and $D_{\nu}^1=\lim_{\lambda\uparrow 1}D_{\nu}^\lambda$. The following theorem indicates the existence of the required stopping time $\tau^*$. 

\begin{theorem}
\label{Th:2} %Suppose that $G$ has right-continuous paths almost surely and
%condition (\ref{rts_eq:1}) holds. Then
%
%\begin{itemize}

%\item[ (i)] $D_{\nu}^{\lambda}, \lambda\in (0, 1)$ is the optimal stopping time
%of $\lambda$;

%\item[ (ii)]
If $G_t$ is quasi-left continuous, then\\
 (1) $D_{\nu }^{1}$ is optimal for the
stopping problem (\ref{rts_eq:3}), that is,
$
Z_{\nu}=\E [G_{D_{\nu}^{1}}|{\cal F}  _{\nu}]\ a.s.
\label{rts_eq:22}$,
%\end{itemize}
\\
 (2) $
D_{\nu}^{1}=\operatorname{essinf}\{\tau\in{\cal M}_{\nu}:Z_{\tau}=G_{\tau}\}=\min\{t\geq \nu (\omega): (\omega, t)\in{\cal M}, Z (\omega, t)=G (\omega, t)\}$
a.s.,  and\\
 (3) $Z_t$ is also
quasi-left-continuous.

\end{theorem}

\subsection{Gittins index process}\label{GittinsIndexProcessSection}
For the instant reward rate process $X_t$ and an arbitrary stochastic processes $q=\{q_t\}$ that is ${\cal F}$-adapted, pathwise right continuous, nonincreasing, bounded and nonnegative, introduce a discounted gain process %$G_q (t;m)$ obtained by receiving a payment flow at rate $X_{u}$ from time $0$ to $t$ and a constant rate $\beta m$ from $t$ onwards, both multiplied by $q_u$, i.e.,
\begin{equation}\label{gain process}
G_q (t;m)=\int_0^t e^{-\beta
u}q_uX_{u}du+\beta m \int_t^\infty e^{-\beta u}q_udu, t\in [0, \infty].
\end{equation}
Note that $q_t\equiv 1$ gives the well-known gain process with retirement option
$G (t;m)=\int_0^t e^{-\beta
u}X_{u}du+m e^{-\beta t}, t\in [0, \infty]$, which was introduced by Whittle (1980).
To any finite ${\cal F}  $-stopping time $\eta$, associate  a class of optimal stopping problems
\begin{align}\label{definition_V}
V_q (\eta, m)=&\operatorname{esssup}_{\tau\in{\cal M}_{\eta}}\E \bigg[
\left.\int_\eta^\tau\text e^{-\beta (u-\eta)}q_uX_{u}du+\beta m\int_\tau^\infty e^{-\beta (u-\eta)}q_udu%
\right|{\cal F}  _\eta\bigg],
\end{align}
 indexed by $m\in[0, \infty)$, indicating the optimal expected rewards from $\eta$ onwards.
Then, for every fixed $m\in[0,\infty)$, the optimal stopping time theory  reviewed in Section \ref{rst_e} can be translated for $V_q (\eta, m)$ to:
\begin{enumerate}
\item[ (1).] The process $Z_q (t, m)=\int_0^{t} e^{-\beta u}q_uX_{u}du+e^{-\beta
t}V_q (t, m), t\in [0, \infty]$ is a quasi-left-continuous supermartingale.

\item[ (2).] The feasible stopping time
\begin{align}\label{sigma_0}
\sigma_{\eta} (m)=&
\operatorname{essinf}\{\tau\in {\cal M}_{\eta}:Z_q (\tau;m)= G_q (\tau;m)\}\nonumber\\
=&\operatorname{%
essinf}\left\{\tau\in {\cal M}_{\eta}:V_q (\tau;m)=
 \beta m\E \bigg[\left. \int_\tau^\infty e^{-\beta (u-\tau)}q_udu%
\right|{\cal F}  _\eta\bigg]
\right\}\in{\cal M}_{\eta}
\end{align} is an optimal
solution for $V_q (\eta;m)$.

 \item[ (3).] $\{Z_q (\tau;m):\tau \hbox{ is } {\cal F}\hbox{-stopping time satisfying }\eta\leq\tau\leq \sigma_{\eta} (m)\}$ is a martingale family.
\end{enumerate}
%   Define accordingly a random variable by
%   \begin{equation}
%  \varphi (\eta, m)=\operatorname*{esssup}_{\tau\in{\cal M}_\eta}\E \left[\left.\int_\eta^\tau e^{-\beta u} (X_{u}-\beta m)du\right|{\cal F}  _\eta\right].
%   \end{equation}
%   Of course,  $\varphi (\eta, m)$ is nonincreasing, convex  in $m$ with $\varphi (\eta, 0)>0$.
%    Recall that there exists an nondecreasing sequence
%   $\{\varphi_n (\eta, m)\}$ such that $\varphi_n (\eta, m)$ is continuous in $m$ for each $n$ and $\varphi_n (\eta, m)\uparrow \varphi (\eta, m)$. By the Dominated Convergence Theorem,  $\varphi (\eta, m)$ is also continuous in $m$.
Moreover, for any finite $\eta\in{\cal M}$ and $m\in[0, \infty)$, write
\begin{align}\label{rsmb_Gittins cont1}
\varphi_q (\eta, m)=&\operatorname{%
esssup}_{\tau\in{\cal M}_{\eta}}\E \left[\left.\int_{\eta}^{\tau}e^{-%
\beta u}q_u (X_u-\beta m)du\right|{\cal F}  _\eta\right].
\end{align} It is then immediate that
\begin{equation}\label{MBP_eq03.10}
\varphi_q (\eta, m)=e^{-\beta \eta}\left[V_q (\eta;m)-\beta m\E\left (\left.\int_\eta^\infty e^{-\beta (u-\eta)}q_udu\right|{\cal F}_\eta\right)\right].
\end{equation}%

\begin{remark}\label{rmk2.1}
Given a stopping time $\eta $,  owing to the the esssup operation, even though for any  couple of nonnegative numbers $m_1<m_2$ and $\lambda\in[0, 1]$, it
is clear that
\begin{align*}
&\Pr\left\{\omega: \varphi_q (\eta, m_1)\geq \varphi_q (\eta, m_2)\hbox{ and }  \varphi_q (\eta, \lambda
m_1+ (1-\lambda)m_2)\geq\lambda \varphi_q (\eta, m_1)+ (1-\lambda) \varphi_q (%
\eta, m_2)\right\}\\
&=1,
\end{align*}
 definition (\ref{rsmb_Gittins cont1}) does not necessarily ensure
pathwise monotonicity and convexity of $ \varphi_q (\eta, m)$ in  $m$. This difficulty can be overcome by a procedure as follows. First, order the rationals
arbitrarily as $Q=\{r_1, r_2, \dots, r_n\dots\}$ and write $Q_n=\{r_1, r_2, %
\dots, r_n\}$. Let $\Omega_1=\Omega$. For $n\geq2$, denote
$
\Omega_n=\left\{\omega:  \varphi_q (\eta, \cdot)
\hbox{ is nonincreasing and
convex on }Q_n\right\}.
$
Then $\Omega_n$ {is decreasing in }$n$ and  $\Pr (\Omega_n) =1 %
\hbox{ for all } n\geq 1. $
%  Iteratively, for any $n$, find $\underline{r}_n$ and $\overline{r}_n$  in $\{r_1, r_2, \dots, r_n\}$ such that $\underline{r}_n<r_{n+1} < \overline{r}_n$. Define $\Omega_{n+1}=\{w:  \varphi_q (\eta, \underline{r}_n)\geq  \varphi_q (\eta, r_{n+1})\geq \varphi_q (\eta, \overline{r}_n)\}$.
Let $\tilde\Omega:=\bigcap_{n=1}^\infty\Omega_n$ such that $\Pr (\tilde\Omega)=1, $
and for every $\omega\in\tilde\Omega$, $ \varphi_q (\eta, m)$ is decreasing along
set $Q$. For the other (real) numbers $m$, take $ \varphi_q (\eta, m;%
\omega)$ as the limit of $ \varphi_q (\eta, r;\omega)$ along $Q$, so
that $ \varphi_q (\eta, m;\omega)$ defined as such is a decreasing and convex function of $%
m$ for every $\omega\in\tilde\Omega$. That is, we get a version of $ \varphi_q (\eta, r;\omega)$ that  is pathwise decreasing and convex in $m$ almost surely.  We will thoroughly work with this version of $%
 \varphi_q (\eta, m)$.
\end{remark}

%By Lemma \ref{proposition_wxy_1}, for any two ${\cal F} $%
%-stopping times $\sigma$ and $\tau$, $\varphi_q (\sigma, m)=\varphi_q (\tau, m)$
%a.s. on $\{\sigma=\tau\}$ for each fixed $m$
%, such that we can define
%arm-specified Gittins indices as below
%.
The following is a  fundamental property of $\sigma (m)$.

\begin{lemma}
\label{rsmb_lem_1} Given a stopping time $\eta\in{\cal M}$, $%
\sigma_\eta (m)$ is nonincreasing and right-continuous in $m$.
\end{lemma}

\begin{proof}
The monotonicity of $\sigma_\eta (m)$ follows from the fact that $$\varphi_q (\sigma_\eta (m_2);m_1)\leq
\varphi_q (\sigma_\eta (m_2);m_2)\leq0 \hbox{ for }m_1>m_2,$$ so that
%\begin{equation}\label{monotone}
 %\varphi_q (\sigma_\eta (m_2);m_1)=0 \hbox{ if }m_1>m_2.
%\end{equation}
%Therefore,
$\sigma_\eta (m_1)=\operatorname{essinf}\{\tau\in{\cal M}_\eta:\varphi_q (\tau;m_1)\leq0\}\leq
\sigma_\eta (m_2)$.
For the right-continuity of $\sigma_\eta (m)$ in $m$, consider a decreasing  sequence $
\delta_n\downarrow0$ of real numbers. By the monotonicity above, the sequence $\sigma_\eta (m+\delta_n)$ is a nondecreasing sequence dominated by $\sigma_\eta (m)$. Then there exists $\sigma_*\in {\cal M}_\eta$ such that $\sigma_*=\lim_{n\to\infty}\sigma_\eta (m+\delta_n)\leq\sigma_\eta (m)$. On the other hand, thanks to
the quasi-left-continuity of $\varphi_q$ (implied by that of $Z$, cf.  Theorem \ref{Th:2} (3)) and  the fact that $\varphi_q (\sigma_\eta (m+\delta_l);m+\delta_k)\leq0$   for any $l>k$, we see that $%
\varphi_q (\sigma_*;m+\delta_k)=\lim_{l\to
\infty}\varphi_q (\sigma_\eta (m+\delta_l);m+\delta_k)\leq0$. Hence,  the continuity of $\varphi_q (\sigma_*, m) $ in $m$ implies  that
$\varphi_q (\sigma_*;m)=\lim_{k\rightarrow\infty}\varphi_q (\sigma_*;m+\delta_k)\leq0$, which in turn implies $\sigma_*\geq \sigma_\eta (m)$. Consequently, $\sigma_*=\sigma_\eta (m)$, that is, $\lim_{n\to\infty}\sigma_\eta (m+\delta_n)=\sigma_\eta (m)$. 

This completes the proof.
\end{proof}

Thanks to this lemma, with a procedure similar to Remark \ref{rmk2.1},  we can work with the version of $\sigma_\eta (m)$ that is nonincreasing and right continuous in $m$ for every $\omega\in\Omega$, so that we can speak of its pathwise inverse
\begin{equation}\label{inverse of sigma}
\underline{M}^q_\eta (t)=\left\{\begin{array}{ll}
\sup\{m\geq 0:\sigma_{\eta} (m)>t\}, &t\geq\eta, \\
\infty, & 0\leq t<\eta\end{array}\right.
\end{equation}
and write particularly
\begin{equation}\label{mbp_eq03.14}
{M}^q_\eta=\underline{M}^q_\eta (\eta) \hbox{ and } \underline{M}^q (t)=\underline{M}^q_0 (t).
 \end{equation}
 The following lemma explains what  these quantities indicate and states that ${M}_\eta:={M}^1_\eta$  is  a direct extension of the classical Gittins index to the  setting with restricted stopping times.

 \begin{lemma}\label{lmgittins}
  Given  $\eta\in {\cal M}$, the following properties hold for the stochastic process $\{M_\eta^q (t)\}$: \hfill\\ (a). $\underline{M}^q_\eta (t)$ is ${\cal F}$-adapted.\\
 (b). ${M}^q_\eta=\inf\{m>0:\varphi_q (\eta, m)\leq0\}$.\\
 (c). $
\underline{M}^q_\eta (\rho)=\operatorname*{essinf}_{\tau\in \cal M_{\eta}, \tau\leq\rho} M_\tau^q
$ for $\rho\in {\cal M}_\eta$.\\
 (d). $
\beta M^q (\eta)=\operatorname{esssup}_{\tau>\eta, \tau\in{\cal M}}\frac{\E %
[\int_{\eta}^{\tau}e^{-\beta u}q_uX_{u}du|{\cal F}  _{\eta}]}{\E %
[\int_{\eta}^{\tau}e^{-\beta u}q_udu|{\cal F}  _{\eta}]}$   .
 \end{lemma}
 \begin{proof}
 (a). For any finite $m>0$ and $t\in[0, \infty)$, if follows that
\begin{align}\label{mbp_eqA3.13}
&\{\omega:\underline{M}^q_\eta (t)>m\}\nonumber\\=&\{\omega:\eta>t\}\cup\{\omega:\sigma_\eta (m)>t\hbox{ and }\eta\leq t \}\nonumber\\
=&\{\omega:\eta>t\}\cup\notag\\
&\left\{\omega: V_q (u, m)>\beta m\E\left[\left.\int_u^\infty e^{-\beta (s-u)}q_sds\right|{\cal F}_u \right]\hbox{ for all }
 (u, \omega)\in{ ([\eta, t]\times\omega)\cap {\cal M} (\omega)}\right\},
\end{align}
where the first equality is a straightforward result of definition \eqref{inverse of sigma} and the second  from equality \eqref{sigma_0}. Note that the first equality implies the adaptedness of $\{\underline{M}^q_\eta (t)\}$, i.e., $\underline{M}^q_\eta (t)\in{\cal F}_t$ for all $t\in\mathbb{R}_+$. This proves (a).

 (b). For $\eta\in{\cal M}$, it is clear that
 \begin{align}%\label{mbp_eq03.14}
{M}^q_\eta=&\inf\left\{m\geq0:V_q (\eta, m)=\beta m\E\left[\left.\int_\eta^\infty e^{-\beta (s-\eta)}q_sds\right|{\cal F}_\eta \right]\right\}
=\inf\{m>0:\varphi_q (\eta, m)\leq0\}.
 \end{align}
%Then, because  $\{m\geq0:V (\eta, m)=m\}=[M (\eta), \infty)$ (cf. the first equivalence in \eqref{mbp_eqA3.13} ), by (\ref{MBP_eq03.10}),
%\begin{equation}\label{mbp_eq03.15}
%{M} (\eta)=\inf\{m>0:V (\eta;m)=m\}=\inf\{m>0:\varphi (\eta, m)=0\}.
%\end{equation}
%That is, the definition of ${M}^q_\eta$ obtained from (\ref{mbp_eq03.14}) by taking $q_t\equiv1$ is equivalent to that in Definition \ref%
%{rsmb_def3.1}.
%Thus, taking $\rho=\eta$ in (\ref{mbp_eqA3.13}) leads to
%\begin{equation}\label{M_eta}
%m<{M} (\eta) \iff\sigma_\eta (m)>\eta \iff V (\eta, m)>m.
%\end{equation}

 (c). Note that,  by (\ref{mbp_eqA3.13}), for $t\geq \eta (\omega)$,
\begin{align*}
\underline{M}^q_\eta (t)>m \iff &V_q (u, m)>\beta m\E\left[\left.\int_\eta^\infty e^{-\beta (s-\eta)}q_sds\right|{\cal F}_\eta \right] \hbox{ for all }u\in [\eta, t]\cap {\cal M} (\omega)\\
\iff &M_u^q>m\hbox{ for all }u\in [\eta, t]\cap {\cal M} (\omega).
\end{align*}
 That is,
$
 \underline{M}^q_\eta (t)=\inf_{\eta\leq u\leq t, (u, \omega)\in {\cal M}}M_u^q.
 $
 Re-expressing this in terms of stopping times leads to the desired equality
$
\underline{M}^q_\eta (\rho)=\operatorname*{essinf}_{\tau\in {\cal M}_{\eta}, {\tau\leq \rho}} M_\tau^q
$ for $\eta\in {\cal M}$ and  $\rho\in {\cal M}_\eta$.

 (d). It is obvious that $V_q (\eta;m)=V^+_q (\eta;m)\bigvee \beta m\E \left[ \left.\int_\eta^{\infty} e^{-\beta (u-\eta)}q_u
du\right|{\cal F}  _\eta\right]$  for $\eta\in {\cal M}$, where
\begin{equation}
V^+_q (\eta, m)=\operatorname{esssup}_{\tau\in{\cal M}, \tau>\eta}\E \bigg[
\left.\int_\eta^\tau\text e^{-\beta (u-\eta)}q_uX_{u}du+\beta m\int_\tau^\infty e^{-\beta (u-\eta)}q_udu%
\right|{\cal F}  _\eta\bigg].
\end{equation}
The assertion in (d) thus follows from the equivalence
\begin{align*}
V_q (\eta;m)\leq\beta m\E \left[ \left.\int_\eta^{\infty} e^{-\beta (u-\eta)}q_u
du\right|{\cal F}  _\eta\right]
 \iff& V^+_q (\eta;m)\leq\beta m\E \left[ \left.\int_\eta^{\infty} e^{-\beta (u-\eta)}q_u
du\right|{\cal F}  _\eta\right]\\
 \iff &\beta m
\geq \operatorname{esssup}_{\tau>\eta, \tau\in{\cal M}}\frac{\E %
[\int_{\eta}^{\tau}e^{-\beta u}q_uX_{u}du|{\cal F}  _{\eta}]}{\E %
[\int_{\eta}^{\tau}e^{-\beta u}q_udu|{\cal F}  _{\eta}]}.
\end{align*}
%Consequently,
%\[
%\beta M^q (\eta)=\operatorname{esssup}_{\tau\in{\cal M}_{\eta+}}\frac{\E %
%[\int_{\eta}^{\tau}e^{-\beta t}q_uX_{u}du|{\cal F}  _{\eta}]}{\E %
%[\int_{\eta}^{\tau}e^{-\beta t}q_udu|{\cal F}  _{\eta}]}, \]
%That is, the definition of ${M}_\eta$ obtained from (\ref{mbp_eq03.14}) by taking $q_t\equiv1$ is equivalent to that in Definition \ref{rsmb_def3.1}.
The proof is thus completed.
\end{proof}

The following lemma establishes a crucial expression for $\E \left[\int_\eta^{\infty}e^{-\beta t}q_tX_{t}dt\right]$ by means of the right derivative of $V_q (\eta, m)$ with respect to $m$.
\begin{lemma}
\label{rsmb_th_2} For any stopping time $\eta\in{\cal M}$,  $ V_q (\eta;m)$
is increasing in $m$ with right-hand derivative
\begin{equation}\label{derivative}
\frac{\partial^+V_q (\eta;m)}{\partial m}%
=\beta \E \left[\left.\int_{\sigma_\eta (m)}^\infty e^{-\beta (u-\eta)}q_udu\right|{\cal F}  _\eta
\right].
%\E \big[e^{-\beta (\sigma_\eta (m)-\eta)}|{\cal F}  _\eta\big]
\end{equation}
As a result,
%for $\eta\in {\cal M}$,
\begin{align}\label{representation}
\E \left[\int_\eta^{\infty}e^{-\beta u}q_uX_{u}du\right]=&\beta\E \left[\left. \int_\eta^{\infty} e^{-\beta u}q_u
\underline{M}^q_\eta (u)du\right|{\cal F}  _\eta\right].
%\hbox{,  and } \nonumber\\
 %\E \left[\int_0^{\infty}e^{-\beta u}q_uX_{u}du\right]=&\beta\E \left[%
%\int_0^{\infty} e^{-\beta u}q_u\underline{M}^q (u) du\right].
\end{align}
% where $\underline{M}^q_t =\underline{M}^q_0 (t).$
\end{lemma}
\begin{proof}
The monotonicity of $V_q (\eta, m)$ in $m$ is straightforward and we first examine  equality \eqref{derivative}.
%
%Introduce a new gain process $H_t=G_{t+}$ and companion filtration ${\cal A}_t={\cal F}_{t+}$, where $t+=\operatorname{essinf}\{\tau>t: \tau\in{\cal M}\}$. Write $\tilde V_q (\eta, m)=\operatorname{esssup}_{\tau\in {\cal M}_\eta}\E\left[\left. H_\tau\right|{\cal A}_\eta\right]$. Then $V_q (\eta, m)=\E[\tilde V_q (\eta, m)|{\cal F}_\eta]$.
For $\delta>0$, Theorem \ref{rts_Th1} (b) and Lemma \ref{rsmb_lem_1} simply state that $Z_q (\eta;m)=\E \big[Z_q (\sigma_\eta (m+\delta);m)|{\cal F}  _\eta\big]$, so that
\begin{align*}
 V_q (\eta;m)=&\E \left[\left.\int_\eta^{%
\sigma_\eta (m+\delta)}e^{-\beta (u-\eta)}q_uX_{u}du+
e^{-\beta (\sigma_\eta (m+\delta)-\eta)}V_q (\sigma_\eta (m+\delta);m)\right|{\cal F}  _\eta\right]\\
\geq&\E \left[\left.\int_\eta^{%
\sigma_\eta (m+\delta)}e^{-\beta (u-\eta)}q_uX_{u}du+\beta m\int_{%
\sigma_\eta (m+\delta)}^\infty e^{-\beta (u-\eta)}q_udu\right|{\cal F}_\eta\right]\\
=&V_q (\eta;m+\delta)-\beta \delta\E \left[\left.\int_{\sigma_\eta (m+\delta)}^\infty e^{-\beta (u-\eta)}q_udu\right|{\cal F}  _\eta\right].
 \end{align*}
Consequently,
\begin{equation}\label{mbp_eq03.11}
V_q (\eta;m+\delta)-V_q (\eta;m) < \beta \delta\E \left[\left.\int_{\sigma_\eta (m+\delta)}^\infty e^{-\beta (u-\eta)}q_udu\right|{\cal F}  _\eta\right].
\end{equation}
 On the other hand, the relationship $$Z_q (\eta;m+\delta)=\E \big[Z_q (\sigma_\eta (m+\delta);m+\delta)|{\cal F}  _\eta]\geq\E \big[Z_q (\sigma_\eta (m);m+\delta)|{\cal F}  _\eta], $$ which is obtained from the supermartingale property of $Z_q (t;m+\delta)$, implies that
  \begin{align*}
  V_q (\eta;m+\delta)\geq &\mathrm{%
E}\left[\left.\int_\eta^{ \sigma_%
\eta (m)}e^{-\beta (u-\eta)}X_{u}du+
e^{-\beta (\sigma_\eta (m)-\eta)}V_q (\sigma_\eta (m);m+\delta)\right|{\cal F}  _\eta\right]\nonumber\\
\geq&V_q (\eta;m)+\beta \delta\E \left[\int_{ \sigma_\eta (m)}^\infty e^{-\beta (u-\eta)}q_udu|{\cal F}  _\eta
\right].
\end{align*}
Hence,
\begin{equation}\label{mbp_eq03.12}
V_q (\eta;m+\delta)-V_q (\eta;m)\geq \beta \delta\E \left[\int_{ \sigma_\eta (m)}^\infty e^{-\beta (u-\eta)}q_udu|{\cal F}  _\eta
\right].
\end{equation}
Combining (\ref{mbp_eq03.11}) with (\ref{mbp_eq03.12}) and letting $\delta\rightarrow0+$ lead to  the desired equality \eqref{derivative}.
%\begin{equation} \frac{\partial^+V_q (\eta;m)}{\partial m}%
%=\beta \E \left[\left.\int_{ \sigma_\eta (m)}^\infty e^{-\beta (u-\eta)}q_udu\right|{\cal F}  _\eta
%\right].
%%\E \big[e^{-\beta (\sigma_\eta (m)-\eta)}|{\cal F}  _\eta\big]
%\end{equation}

\bigskip
By \eqref{derivative} and the equality $V_q (\eta;M^q_\eta)-V_q (\eta;0)
=\int_0^{M^q_\eta}\frac{\partial^+V (\eta;m )}{\partial m }dm$, it follows that
\begin{align*}
V_q (\eta;M^q_\eta)-V_q (\eta;0)
%&=\int_0^{M^q_\eta}\frac{\partial^+V (\eta;m )}{\partial m }dm  \\
&=\beta\int_0^{M^q_\eta}  \E \left[\left.\int_{\sigma_\eta (m)}^\infty e^{-\beta (u-\eta)}q_udu\right|{\cal F}  _\eta
\right]dm  \\
&=M^q_\eta\beta\E \left[\left.\int_\eta^{\infty}e^{-\beta (u-\eta)}q_udu\right|{\cal F}_\eta\right]-\beta\E \left[\left.\int_0^{M^q_\eta}\int_\eta^{\sigma_\eta (m)}e^{-\beta (u-\eta)}q_ududm\right|{\cal F}_\eta\right].
\end{align*}
Noting that $V_q (\eta;M^q_\eta)=\beta M^q_\eta\E \left[\left.\int_\eta^{\infty}e^{-\beta (u-\eta)}q_udu\right|{\cal F}_\eta\right]$, it is immediate that
\begin{equation}\label{V_q}
V_q (\eta;0)=\beta\E \left[\left.\int_0^{M^q_\eta}\int_\eta^{\sigma_\eta (m)}e^{-\beta (u-\eta)}q_ududm\right|{\cal F}_\eta\right].
\end{equation}
 Due to the relationship
\[
\{ (m, u):\eta\leq u\leq\sigma_\eta (m), 0\leq m\leq M^q_\eta\}=\{ (m, u): 0\leq m\leq \underline{M}^q_\eta (u), \eta\leq u<\infty\},
\]
it follows by interchanging the integrations in \eqref{V_q} that
\[
V_q (\eta;0)=\beta\E \big[ \int_\eta^{\infty} e^{-\beta (u-\eta)}
q_u\underline{M}^q_\eta (u)du|{\cal F}  _\eta\big].
\]
Thus the desired equality in \eqref{representation} follows.% The second one in \eqref{representation} is obvious by taking $\eta=0$.
\end{proof}

%%%%%%%%%%%%%%%%%%%%%%%%%%%%%%%%%%%%%%%%%%%%%%%%%%%%%%%%%%%%%%%%%%%%%%%%%%%%%%%%%

We will need to treat the case where  one has an extra $\sigma$-algebra ${\cal G}^{\prime}$ that is independent of the filtration ${\cal F}$. This introduces a new
filtration ${\cal G}=\{{\cal G}_{t}\}$ by ${\cal G}_{t}={\cal F%
}_{t}\vee{\cal G}^{\prime}$,  generally called an \textit{initial
enlargement} (or \textit{augmentation}) of ${\cal F}  $ by ${\cal G}%
^{\prime}$. Denote the set of all ${\cal G}$-stopping times taking values a.s. in $\cal M$ by $\cal M^G$ and those taking values in $\cal M$ and larger than or equal to $\tau$ by ${\cal M}^{\cal G}_\tau$.
Consider the setting in which
\begin{enumerate}
\item $X_t$ is $\cal F$-adapted and
\item $q_t$ is ${\cal G}$-adapted, almost surely right continuous, and right decreasing at such time $t$ with $ (t, w)\in{\cal M}$.
\end{enumerate}
%For any $\eta\in{\cal M}$, denote by $M_\eta$ the Gittins index of $X$ with $q_t\equiv1$ and with respect to the filtration $\cal F$ at $\eta$.
%Similarly, for  ${\cal G}$-stopping time $\tilde\eta$,  introduce
%\begin{align}
%\varphi^q ({\tilde\eta}, m)&=\operatorname{esssup}_{{\tilde{\rho}}\in {\cal M}_{%
%\tilde{\eta}}}\E \left[\int_{\tilde{\eta}}^{\tilde{\rho}}e^{-\beta
%t}q_u [X_{u}-\beta m]du\big|{\cal G}_{\tilde{\eta}}\right],  \notag \\
%M^q (\tilde{\eta})&=\operatorname{inf}\{m: \varphi^q (\tilde{\eta}, m)\leq0\} \hbox{ and }
%\underline{M}^q_t =\operatorname{essinf}_{\tilde{\eta}\in {\cal M}_{[0, t]}}
%M^q (\tilde{\eta})  \notag
%\end{align}
%where ${\tilde{\rho}}\in {\cal M}_{\tilde{\eta}}$ means ${\tilde{\rho}}\in%
%{\cal \tilde{T}}$ and ${\tilde{\rho}}\in {\cal M}_{\tilde{\eta}}$. The
%same argument leads the following conclusion
%\begin{equation}
%\E \big[\int_0^{\infty}e^{-\beta t}q_t X_{t}dt\big]=\E \big[%
%\int_0^{\infty}\beta e^{-\beta t}q_t \underline{M}^q_t dt\big].
%\end{equation}
%%%%%%%%%%%%%%%%%%%%%%%%%%%%%%%%%%%%%%%%%%%%%%%%%%%%%%%%%%5
%Recall that $Z (\tau, m)=\operatorname{esssup}_{\nu\in{\cal M}_{\tau}}E[G (\nu, m)|%
%{\cal F}  _{\tau}]$, where $\nu\in{\cal M}_{\tau}$ means arbitrary $ (%
%{\cal F}  _t)$-stopping time $\nu$ evaluated in ${\cal M}_{\tau}$. The
%optimal stopping times for $Z (\tau, m)$ states that
%\begin{equation}
%Z (\tau, m)=\E [G (\sigma_\tau (m) , m)|{\cal F}  _{\tau}]\quad\text{%
%where}\quad\sigma_\tau (m)  =\operatorname{essinf} \{\nu\in{\cal M}%
%_{\tau}:Z (\nu, m)=G (\nu, m)\}.  \label{rsmb_a_6}
%\end{equation}
Under the augmented filtration $\cal G$, taking the right continuous version of $Z (t, m)$, we can extend the notation $Z (\tau, m)$ to  any $\cal G$-stopping times $\tau$ by $Z (\tau, m)=Z (\tau (\omega), m)$. Define a new optimization problem $\tilde{Z} (\tau, m)=\operatorname{esssup}%
_{\nu\in{\cal M}_{\tau}^{\cal G}}E[G (\nu, m)|{\cal G}_{\tau}]$.
%, where the essential supremum is taken over all ${\cal G}$-stopping times evaluated
%in ${\cal M}_{\tau}$ .
Then it is straightforward that $Z (\tau, m)=\tilde{Z} (\tau, m)$ for any ${\cal G}  $-stopping time $\tau$, which states that,
regardless of the enlargement of the domain of stopping times by initially introducing extra information, the optimal stopping problem basically remains if the additionally obtained information is independent of
the original information filtration $\cal F$ and $X_t$ is $\cal F$-adapted.
The following lemma holds for any $\cal G$-adapted, right continuous $q_u$ that is right decreasing  only when $u\in\cal M$.
\begin{lemma}
\label{rsmb_lem_3} Let $\tilde{X}_t$ be an arbitrary ${\cal F}$-adapted process and $q_t$  be ${\cal G}$-adapted, right continuous, and right decreasing at time $t\in{\cal M}$. Then, for any $\cal F$-stopping times $\eta\in {\cal M}$, the inequality $\operatorname{esssup}_{\nu\in{\cal M}%
_{\eta}}E[\int_\eta^{\nu}e^{-\beta t}\tilde{X}_tdt|{\cal F}_{\eta}]\leq 0$ implies
$
\E \big[\int_\eta^{\infty}e^{-\beta t}q_t\tilde{X}_t dt|{\cal G}%
_{\eta}\big]\leq 0.
$

\end{lemma}

\begin{proof}
%Note that $\zeta (u)-u$ is left continuous.
Introduce the right continuous inverse $q^{-1} (s)=\min\{u:q_u\leq s\}=\max\{u:q_u>s\}$. Then, for any $s$, $q^{-1} (s)$ is a $\cal G$-stopping time because $\{\omega: q^{-1} (s)\leq u\}=\{\omega: q_u \leq s\}\in{\cal G}_u$. In addition,
\begin{enumerate}
\item for any $t\geq \eta$, the relationship $s<q_t (\leq q_\eta)$ implies $q^{-1} (s)>\eta$  and
\item for any $\omega$,  $q^{-1} (s)\in {\cal M} (\omega)$ because $q_u$ is right decreasing only when $u\in {\cal M}$.
\end{enumerate} These two points actually further state that,  for any $s\in[0, 1)$, $q^{-1} (s)$ is a $\cal G$-stopping time  in $\cal M_\eta^G$. Note that $q_\eta\in {\cal G}_\eta$. Therefore,
\begin{align*}
\E \left[\int_\eta^{\infty}e^{-\beta t}\tilde{X}_tq_t dt\bigg|{\cal G}%
_{\eta}\right]=&\E \left[\int_\eta^{\infty}e^{-\beta t}\tilde{X}_t\int^{q_t}_0dsdt\bigg|{\cal G}%
_{\eta}\right]
=\int_0^{q_\eta}\E \left[\int_\eta^{q^{-1} (s)}e^{-\beta t}\tilde{X}_tdt\bigg|{\cal G}_{\eta}\right]ds
\leq 0,
\end{align*}
This completes the proof.
\end{proof}

With Lemma \ref{rsmb_lem_3}, setting $\tilde{X}_t=X_t-\beta m$ and replacing $q_u$ in  Lemma \ref{rsmb_lem_3} by  $\tilde q_u=I_{[\eta, \tau)}q_u$, it is immediate that $\varphi (\eta, m)\leq0$ implies  $\varphi_q (\eta, m)\leq0$. Consequently,
\begin{equation}\label{rsmb_cor_2}
M^q_\eta=\inf\{m>0:\varphi_q (\eta, m)\leq0\}\leq M_\eta=\inf\{m>0:\varphi (\eta, m)\leq0\} \hbox{ for all }\eta\in{\cal M}.
\end{equation}
%\begin{corollary}
%\label{rsmb_cor_2} $M^q (\tau)\leq M (\tau)$ a.s. for  $\tau\in {\cal M}$, where $q_t $ is defined in \eqref{rsmb_def_3}.
%\end{corollary}
%%%%%%%%%%%%%%%%%%%%%%%%%%%%%%%%%%%%%%%%%%%%%%%%%%%%%%%%55
Now let $T (t)$ be a generic component of a policy (i.e., $T^k (t)$ in the policy formulation with some $k\in\{1, 2, \dots, d\}$). We address the particular choice
\begin{equation}\label{q_u}
q_u=\exp[-\beta (\zeta (u)-u)],
\end{equation}
where
\begin{equation}\label{zeta}
\zeta (u)=\inf\{t:T (t)> u\}\ (\hbox{ hence, }T (t) <  u\iff \zeta (u) <  t)
\end{equation} is the right continuous inverse of $u=T (t)$, indicating the calendar time of the system at which the current arm, which has been operated for $u$ units of time, is to be selected for further operation, so that $\zeta (u)-u$ is the time spent on other arms and thus is also nondecreasing in $u$. Clearly, the particular $q_u$ in \eqref{q_u} is ${\cal G}$-adapted, right continuous, and right decreasing at time $u\in{\cal M}$. The following lemma gives an bound for the expected discounted reward from a single arm under any policy.
%
%$\zeta (u)$ is a $RCLL$
%strictly increasing process, it is trivial to see, from the technical requirement (4)
%of the restricted allocation policy, that $q_u $ is a $RCLL$ nonincreasing
%process with $q_u =q (D_{u}-) \text{ for }u\in {\cal B}$, $%
%\int_s^{t}e^{-\beta u}X_{u}q_u du=\int_s^{t}e^{-\beta u}X_{u}q_u du$ for $%
%t\geq s, s, t\in {\cal M} $.

\begin{lemma}\label{rsmb_eq_12}Let $T (t)$ be a generic component of a policy. Then
\begin{equation}
\E \left[\int_0^{\infty}e^{-\beta t}X_{T (t)}dT (t)\right]\leq\E \left[\int_0^{\infty}\beta e^{-\beta t}\underline{M} (T (t))dT (t)\right].
\end{equation}
\end{lemma}
\begin{proof}
  First note that, by the definition $q$ in \eqref{q_u},
 \[
 \E \left[\int_0^{\infty}e^{-\beta t}X_{T (t)}dT (t)\right]=\E %
\left[\int_0^{\infty}e^{-\beta\zeta (t)}X_{t}dt\right]
=\E \left[\int_0^{\infty}e^{-\beta t}q_t X_{t}dt\right].
 \] Because $\underline{M} (t)$ and $\underline{M}^q (t)$ are both nonincreasing, $\underline{M}^q (t)\leq \underline{M} (t)$ follows from \eqref{rsmb_cor_2}. Therefore, an application of equality \eqref{representation} indicates that
\begin{align}
\E \left[\int_0^{\infty}e^{-\beta t}X_{T (t)}dT (t)\right]= \E \left[\int_0^{\infty}\beta e^{-\beta t}q_t \underline{M}^q (t)dt\right]
\leq& \E \left[\int_0^{\infty}\beta e^{-\beta t}q_t \underline{M} (t)dt\right].\nonumber
\end{align}
Using again the definition of $q$ leads to
\[
\E \left[\int_0^{\infty}e^{-\beta t}X_{T (t)}dT (t)\right]\leq\E \left[\int_0^{\infty}\beta \exp (-\beta\zeta (t))\underline{M} (t)dt\right]
\nonumber\\
=\E \left[\int_0^{\infty}\beta e^{-\beta t}\underline{M} (T (t))dT (t)\right].
\]
This proves the lemma.
%That is,
%\begin{equation}  \label{rsmb_eq_12}
%\E \big\{\int_0^{\infty}e^{-\beta t}X_{T (t)}dT (t)\big\}\leq \E %
%\big\{\int_0^{\infty}\beta e^{-\beta t}\underline{M}[T (t)]dT (t)\big\}.
%\end{equation}
\end{proof}

\section{Optimal Allocation of RMAB Processes}

\label{rsmb_proof_G_P}
\setcounter{equation}{0}
We are now ready to state and prove the results on the optimal policies for the RMAB processes. The
identifier $k$ of arms has to be added back.

We will see that the solution to this problem is
still the celebrated Gittins index policy,  with Gittins indices generalized as  follows.
Note that for $\eta\in{\cal M}^k$, the Gittins index is the same as $M_\eta^q$ introduced in the previous section (by \eqref{inverse of sigma} and \eqref{mbp_eq03.14}) with $q\equiv 1$, see especially Lemma \ref{lmgittins} (d), whereas for $\eta$ not in ${\cal M}^k$, the definition is new.
%Given arm $k$. For any ${\cal F}^k  $-stopping time $\eta$, define  stopping times
%\begin{equation}\label{eta_M}
%\underline\eta =\operatorname{esssup}\{\tau\in{\cal M}^k :\tau\leq  \eta \}\in{\cal M}^k \hbox{ and }
%\eta+=\operatorname{essinf}\{\tau\in{\cal M}^k :\tau>\eta \}\in{\cal M}^k
%\end{equation} so that $\xi>\eta\iff \xi>\underline\eta$ for any $\cal F$-stopping time $\xi\in{\cal M}^k $.

\begin{definition}
\label{rsmb_def3.1} The Gittins indices $\{M^k_\eta:\eta\hbox{ is a finite }{\cal F}^k\hbox{-stopping time}\}$ of arm $k$ are defined by two steps:\\
%\begin{enumerate}
Step 1. For   ${\cal F}^k$-stopping times $\eta\in{\cal M}^k$, compute
\begin{equation}
M^k_\eta=\operatorname{esssup}_{\tau>{\eta}, \tau\in{\cal M}^k}\frac{\E %
[\int_{\eta}^{\tau}e^{-\beta u}X^k_{u}du|{\cal F}^k _{\eta}]}{\beta\E %
[\int_{\eta}^{\tau}e^{-\beta u}du|{\cal F}^k  _{\eta}]},
%\varphi^k (\tau, m;\omega) (\hbox{abbreviated as  }\varphi^k (\tau, m)) =\operatorname{%
%esssup}_{\rho\in{\cal M}^{k}_{\tau}}\E \big[\int_{\tau}^{\rho}e^{-%
%\beta u} (X^k (u)-\beta m)du|{\cal F}  _\tau^k\big],
\end{equation}
where the essential supremum is taken over the set $\{\tau: \tau> {\eta}\hbox{ and }\tau\in {\cal M}^k\}$ of ${\cal F}  ^k$%
-stopping times.\\
Step 2. For other  ${\cal F}^k$-stopping times $\eta$, find $\underline\eta=esssup\{\tau\leq {\eta}:\tau\in {\cal M}^k\}$ ($\in{\cal M}^k$) and define  $M^k_\eta=
M^k_{\underline\eta}$.
\end{definition}

Since $M_\eta^k$ is defined for all stopping times $\eta$, we can construct an associated process $M^k=\{M_{t}^{k} (\omega):t<\infty,
\}$ to $M^k_\eta$ (called
Gittins
indices process) and all the processes $\{M^k, k=1, 2, \dots, d\}$ serve to select an arm to operate, as will be illustrated later on.
\begin{remark}
Here note that Definition \ref{rsmb_def3.1} unifies and extends all the classical definitions of Gittins indices in discrete time, continuous time and   general discrete time setting to the current RMAB process situation. For example, in the case ${\cal M}^k=\mathbb{N}\cup \{\infty\}$,  $ (M_n^{k})$ is a stochastic sequence of
Gittins indices in integer times, coinciding with what were defined by
Gittins and Jones (1974) and Gittins
 (1979).
\end{remark}
%Also introduce the lower envelope $\underline{M}^k$ of each $M^k$ by
%\begin{equation}
%\underline{M}^k_s (t) (\omega)=\inf_{u\in[s, t]%
%}M^k_u (\omega), \ t\geq 0.
%\end{equation}
%Of course, \textbf{ If ${\cal M}$ is a progressive set } $\underline{M}^k (s, t) (\omega)=\operatorname{essinf}_{\tau\in{\cal M%
%}_{[s, t]}}M^k_{\tau} (\omega), \ t\geq 0.$
%Simplify the notation $\underline{%
%M}^k_0 (t)$ as $\underline{M}^k_t$.
{With the lower envelope $\underline{M}^k (t)=\min_{0\leq u\leq t}M^k (u)$ (see  \eqref{mbp_eq03.14} and Lemma \ref{lmgittins} (c)), let ${\cal M}%
^k_1 (\omega)=\text{closure}\{t: (t, \omega)\in{\cal M}^k (\omega), M^k_t (\omega)=%
\underline{M}^k (t, \omega) \}$ and call the times in the complement of $%
{\cal M}^k_1 (\omega)$  excursion times of $M^k$ from its lower
envelope $\underline{M}^k$.} It is obvious that, for fixed $\omega$, the set of excursion times is a union of countably many open intervals. 

The next presents the definition of (Gittins) index policy.
\begin{definition}
A restricted policy $\hat{T}= (\hat{T}^1, \cdots, \hat{T}^d)$ is a
 (Gittins) index policy if, for each $k$, $\hat{T}^k=\{\hat{T}^k (t)\geq 0\}$
right increases at time $t\geq 0$ only when
\begin{equation}
M^{k}_{\hat{T}^k (t)}=\bigvee_{j=1}^{d}M^{j}_{\hat{T}^j (t)}.
\end{equation}
{and time must be allocated exclusively to a single arm over its excursion
interval without switching.}
\end{definition}

Because ${\cal M}_1^k\subset {\cal M}^k$, it is clear that an index policy satisfies the restrictions on policies. As observed by
Mandelbaum (1987),  index policies need not be unique. The solution to the RMAB process is stated in the theorem below.

\begin{theorem}
\label{rsmb_th_1} Any (restricted) index policy $\hat T= (\hat{T}^1, \cdots, \hat{T}^d)$ is optimal with
the optimal value expressed in terms of the lower envelopes of the
indices as
\begin{equation}\label{optimal}
V=\E \left[\int_0^{\infty}\beta e^{-\beta t}\bigvee_{k=1}^d
\underline{M}^{k} (\hat{T}^k (t))dt\right].
\end{equation}
\end{theorem}

\begin{proof}
Define $\tilde {{\cal F}  }%
_{t}^{k}={\cal F}  _{t}^{k}\bigvee_{j\neq k}{\cal F}  _{\infty }^{j}$. Fix
an arbitrary policy $T$ and let $\zeta^{k} (t)=\inf\{u:T^{k} (u)>t\}$ be the
generalized inverse of $T^{k} (t)$.
%Note that $T^{k} (t)$ is an increasing continuous process, $\zeta^{k} (t)$ is  an
%increasing process with RCLL paths [right-continuous, with limits from the left].
Define
\begin{align} 
v (T)=&\E \left[\sum_{k=1}^d\int_0^{\infty}e^{-\beta t}X^k_{T^k (t)}dT^k (t)%
\right]\hbox{ and }\notag\\ \underbar{\it v} (T)=&\E \left[\sum_{k=1}^d\int_0^{\infty}\beta e^{-\beta t}%
\underline{M}^{k} (T^k (t))dT^k (t)\right] \label{rsmb_eq_18}
\end{align}
to represent respectively the expected values of the original bandit process and a deteriorating bandit process with reward rates $\underline{M}^{k} (t), k=1, 2, \dots, d$ under the same policy $T$.
Note that Lemma \ref{rsmb_eq_12} simply states that
\[
\E \left[\int_0^{\infty}e^{-\beta t}X^k_{T^k (t)}dT^k (t)\right]\leq
\E \left[\int_0^{\infty}\beta e^{-\beta t}\underline{M}^{k} (T^k (t))dT^k (t)\right].
\]
%%%%%%%%%%%%%%%%%%%%%%%%%%%%%%%%%%%%%%%%%%%%%%%%%%%%%%%%%%%%%5
Summing it over all arms $k=1, 2, \dots, d$, it follows that under any policy $T$,
\[
v (T)\leq \underbar{\it v} (T).
\]
Thus, to prove the optimality of an index policy $\hat T$, it suffices to prove $\underbar{\it v} (T)\leq \underbar{\it v} (\hat T)=v (\hat T)$.
This is done by the following Lemmas \ref{lemma_pf_Th_1} and \ref{lemma_pf_Th_2}.
\end{proof}
%%%%%%%%%%%%%%%%%%%%%%%%%%%%%%%%%%%%%%%%%%%%%%%%%%%%%%%%%%%%%%%%
\begin{lemma}\label{lemma_pf_Th_1}
 For any policy $T$ and index policy $\hat T$, $\underbar{\it v} (T)\leq \underbar{\it v} (\hat T).$
\end{lemma}
\begin{proof}
 Under $T$, the total discounted reward $\underbar{\it v} (T)$ for the reward rates $\{\beta\underline{M}^k (t), k=1, 2, \dots, d\}$ is%
\begin{equation}
\underbar{\it v} (T)=\beta\sum\limits_{k=1}^{d}\E\big[\int_{0}^{\infty}e^{-\beta t}\underline{M}^{k} ({T^{k} (t)})d{T^{k} (t)}\big].
%=\sum\limits_{k=1}^{d}\int_{0}^{\infty}e^{-\beta t}\underline{M}^{k} ({T_{t}^{k}})f_{t}^{k}dt.
 \label{E15_bandit_wu}%
\end{equation}
%where the reward rate $\underline{M} ({t}^{k})$ is supposed to be right-continuous and
%$f^k (t)\geq0$ are the almost everywhere derivatives of $T_{t}^{k}, k=1, 2, \ldots d$
%such that $\sum_{k=1}^{d}f^k (t)=1$. Then
%$v (T)\leq v (\hat{T})$, where $T$ is arbitrary and $\hat{T}$ a policy
%following the leading reward rate. This is proved as follows.
Write $h_{u}^{k}=\sup\{t:\underline{M}^{k} ({t})>u\}=\inf\{t:\underline{M}^{k} ({t})\leq u\}$ for the
right-continuous inverse of $\underline{M}^{k} ({t})$, which models the time needed to
operate the arm such that its reward rate falls down to a level no more than
$u$. Because $\underline{M}^k (t)=\underline{M}^k (\underline t)$,  where $\underline t$ is defined as in Definition \ref{rsmb_def3.1}, it is clear that $h_u^k\in{\cal M}^k$. Thus, in order that all $\underline{M}^{k}$ can fall down to level $u$, one needs to spend in total  $\tilde{h}_{u}=\sum_{k=1}%
^{d}h_{u}^{k}$ units of time on the $d$ arms.

We first examine the equality
\begin{equation}
\sum_{k=1}^{d}\hat{T}^{k} (t)\wedge h_{u}^{k}=t\wedge\tilde{h}_{u}
\label{E14_bandit_wu}%
\end{equation}
over the set of  $t$ at which all $\underline{M}^{k} ({\hat{T}^{k} (t)})$, $k=1, 2, \ldots, d$,
are continuous.

Since $\sum_{k=1}^{d}\hat{T}^{k} (t)=t$ and
$\tilde{h}_{u}=\sum_{k=1}^{d}h_{u}^{k}$, it suffices to
show that there exists no pair
$ (k, p)$ of identifiers such that%
\begin{equation}
\hat{T}^{p} (t)>h_{u}^{p}\quad\text{and}\quad\hat{T}^{k} (t)<h_{u}^{k}
\label{wu_bandit_4}%
\end{equation}
if both $\underline{M}^{k} ({\hat{T}^k (t)})$ and
$\underline{M}^{p} ({\hat{T}^p (t)})$ are continuous at point $t$.
We prove it by contradiction.
If (\ref{wu_bandit_4}) were true, then $\underline{M}^{p} ({\hat{T}^p (t)})\leq
u<\underline{M}^{k} ({\hat{T}^k (t)})$. Define%
\[
\tau=\sup\{s:s<t, \underline{M}^{p} (\hat{T}^p (t))\geq \underline{M}^{k} ({\hat{T}^k (s)})%
\}=\inf\{s:s\leq t, \underline{M}^{p} ({\hat{T}^p (t)}) < \underline{M}^{k} (\hat{T}^k (s))\}
\]
with the convention $\sup\varnothing=0$. Because $\underline{M}^{p} ({\hat{T}^p (s)})%
 < \underline{M}^{k} ({\hat{T}^k (s)})$ for all $s\in (\tau, t]$, the feature of $\hat{T}$
following the leader indicates that
\begin{equation}
\hat{T}^p (t)=\hat{T}^p (\tau). \label{wu_bandit_3}%
\end{equation}
If $\tau=0$ then $\hat{T}^ (t)=0\leq
h_{u}^{p}$, contradicting (\ref{wu_bandit_4}). Otherwise, i.e., $\tau>0$, find a sequence of nonnegative sequence $\alpha
_{n}\rightarrow0$ such that $\underline{M}^{p} ({\hat{T}^p (\tau-\alpha_{n})})\geq
\underline{M}^{k} ({\hat{T}^k (\tau-\alpha_{n})})\geq \underline{M}^{k} ({\hat{T}^k (t)})>u$; if
$\underline{M}^{p} ({\hat{T}^p (\tau)})\geq \underline{M}^{k} ({\hat{T}^k (\tau)})$, then $\alpha_{n}$
all take value $0$. Therefore, $\hat{T}^p (\tau-\alpha_{n})<h_{u}^{p}$.
Setting $n\rightarrow\infty$ and using (\ref{wu_bandit_3}) lead to $\hat
{T}^p (t)=\hat{T}^p (\tau)\leq h_{u}^{p}$, contradicting
 (\ref{wu_bandit_4}) again. Thus (\ref{E14_bandit_wu}) is proved.

We now turn to check the desired inequality of the lemma. For
any policy $T$,
\begin{equation}
\sum\limits_{k=1}^{d}\left ( T^k (t)\wedge h_{u}^{k}\right)  \leq
\sum\limits_{k=1}^{d}T^k (t)\wedge\sum\limits_{k=1}^{d}h_{u}^{k}%
=t\wedge\tilde{h}=\sum\limits_{k=1}^{d}\left ( \hat{T}^k (t)\wedge h_{u}%
^{k}\right)  . \label{wu_bandit_5}%
\end{equation}
Thus, by (\ref{E15_bandit_wu}),
\begin{align*}
\underbar{\it v} (T)=&\beta\sum\limits_{k=1}^{d}\E\left[\int_{0}^{\infty}e^{-\beta t}\underline{M}
^{k} ({T^k (t)})dT^k (t)\right]\\
=&\beta\sum\limits_{k=1}^{d}\E\left[\int_{0}^{\infty}e^{-\beta t}\int
_{0}^{\infty}I_{\{0<u<\underline{M}^{k} (T^k (t))\}}dudT^k (t)\right].
\end{align*}
By Fubini's theorem, %
\begin{align*}
\underbar{\it v} (T)=&\beta\sum\limits_{k=1}^{d}\E\left[\int_{0}^{\infty}\int_{0}^{\infty}e^{-\beta
t}I_{\{T^k (t)<h_{u}^{k}\}}dT^k (t)du\right]\\
=&\beta\sum\limits_{k=1}^{d}\E\left[\int_{0}^{\infty
}\int_{0}^{\infty}e^{-\beta t}d\left ( T^k (t)\wedge h_{u}^{k}\right)  du\right].
\end{align*}
Further using the partial integration and the inequality in (\ref{wu_bandit_5}%
) yields%
\begin{align*}
\underbar{\it v} (T)=&\beta^2\E\left[\int_{0}^{\infty}\int_{0}^{\infty}e^{-\beta t}\sum\limits_{k=1}%
^{d}\left ( T^k (t)\wedge h_{u}^{k}\right)  dtdu\right]\\
\leq&\beta^2\E\left[\int_{0}^{\infty
}\int_{0}^{\infty}e^{-\beta t}\sum\limits_{k=1}^{d}\left ( \hat{T}^k (t)\wedge h_{u}^{k}\right)  dtdu\right]\\
=&\underbar{\it v} (\hat{T}).
\end{align*}
This proves the desired result.
\end{proof}
\begin{remark}
An arm is said {\it deteriorating} if its reward rate  is pathwise nonincreasing
in time and a bandit is deteriorating if all its arms are  deteriorating. In this case, the
optimal policy is myopic in the sense that it plays the arms with the highest
immediate reward rate. In fact, this lemma can be generalized as: $v (T)\leq v (\hat T)$ for any restricted deteriorating bandits $\{ (X_t^k, {\cal M}_k), k=1, 2, \dots, n\}$, where ${\cal M}_k$ is the feasible set of stopping times associated with arm $k$.
\end{remark}
\begin{lemma}\label{lemma_pf_Th_2}
For any index policy $\hat T$, $v (\hat T)=\underbar{\it v} (\hat T)$.
\end{lemma}
\begin{proof}
For every arm $k$, introduce two policy-free quantities
\[
D^k_{t}:=D^k_{t} (\omega)=\inf\{u> t:u\in {\cal M}^k_1 (\omega)\}\hbox{ and }
g^k_{t}:=g^k_{t} (\omega)=\sup\{u< t:u\in {\cal M}^k_1 (\omega)\} \text{ for } t>0, \]
 so that, { for every } $t>0\hbox{ and }u>t$,
\begin{equation}\label{lowerenvelop_1}
g^k_u\leq t\iff u\leq D^k_t\\
\end{equation}and
\begin{equation}\label{lowerenvelop_2}
\underbar{M}^k_{D^k_t} (u)=\underbar{M}^k (u).
\end{equation}
%{ and }\[
%{\cal B}_1=\{t\in {\cal M}^k_1 (\omega): D_{t} (\omega)>t\}.
%\]
Note that, under any index policy $\hat{T}= (\hat{T}^1, \cdots, \hat{T%
}^d)$, $\zeta^k (t)=\zeta^k (g^k_t)+t-g^k_t$. See \eqref{zeta} for the definition of $\zeta^k (t)$, where the identifier $k$ was suppressed.
Define $H_k (u)=\zeta^k (u)-u$ and its inverse $H^{-1}_k (s)=\inf\{u: H_k (u)>s\}$, which is the time  spent on arm $k$ when $s$ units of time has been spent on the other $d-1$ arms and an arm other than $k$ is to be selected to operation,  so that $H_k (u)\leq s\iff u\leq H_k^{-1} (s)$. Consequently, $H_k (g^k_u)\leq s\iff H_k^{-1} (s)\geq g^k_u\iff u\leq D_{H_k^{-1} (s)}$. Therefore, the expected reward under the index policy $\hat T$ can be computed by
\[
v (\hat T)=\sum_{k=1}^d\E \bigg[\int_0^{\infty}e^{-\beta \zeta^k (u)}X^k_udu\bigg]=\sum_{k=1}^d\E \bigg[\int_0^{\infty}e^{-\beta (\zeta^k (g^k_u)-g^k_u)}e^{-\beta u}X^k_udu\bigg].\]
A bit algebraic computation gives rise to
\begin{align*}
v (\hat T)=&\beta\sum_{k=1}^d\E \left[\int_0^{\infty}e^{-\beta u}X^k_u\int_{\zeta^k (g^k_u)-g^k_u}^\infty e^{-\beta s}dsdu\right]\notag\\
=&\beta\sum_{k=1}^d \left[\int_0^{\infty}e^{-\beta s}\E\left (\int_0^{D_{H_k^{-1} (s)}}e^{-\beta u}X^k_u du\right)ds\right].
\end{align*}
By the first equality in \eqref{representation} under $\mathcal{\tilde{F}}$ (see Lemma \ref{rsmb_th_2}), it then follows that
\begin{align*}
&\E\left (\int_0^{D_{H_k^{-1} (s)}}e^{-\beta u}X^k_u du\right)\notag\\
=&\E\left[\int_0^{\infty}e^{-\beta u}X^k_u du-\E\left (\left.\int_{D_{H_k^{-1} (s)}}^\infty e^{-\beta u}X^k_u du\right|{\cal \tilde{F}}^k_{D_{H_k^{-1} (s)}}\right)\right]\\
=&\beta\E\left[\int_0^{\infty}e^{-\beta u}\underbar{M}^k (u) du-\E\left (\left.\int_{D_{H_k^{-1} (s)}}^\infty e^{-\beta u}\underbar{M}_{D_{H_k^{-1} (s)}}^k (u) du\right|{\cal \tilde{F}}^k_{D_{H_k^{-1} (s)}}\right)\right]\\
=&\beta\E\left[\int_0^{\infty}e^{-\beta u}\underbar{M}^k (u) du-\E\left (\left.\int_{D_{H_k^{-1} (s)}}^\infty e^{-\beta u}\underbar{M}^k (u) du\right|{\cal \tilde{F}}^k_{D_{H_k^{-1} (s)}}\right)\right]\\
=&\beta\E\left[\int_0^{D_{H_k^{-1} (s)}}e^{-\beta u}\underbar{M}^k (u) du\right],
\end{align*}
where the last equality follows from the equality in \eqref{lowerenvelop_2}. Consequently,
\begin{equation*}
v (\hat T)=\beta^2\sum_{k=1}^{\infty}\E \left[\int_0^{\infty}e^{-\beta s}\int_0^{D_{H_k^{-1} (s)}}e^{-\beta u}\underbar{M}^k (u) duds\right]\\
=\underbar{\it v} (\hat T).
\end{equation*}
This proves the desired equality.
\end{proof}
\section{Conclusions}\label{Conclusions}
By the extended  optimal stopping theory to the  problem with restricted stopping times,  and combining the martingale techniques and Whittle (1980)'s retirement option, we have
developed a general and new framework that generalizes and unifies the theories of multi-armed bandit processes in discrete time, continuous time, which apply also to other mixed settings
 that can occur practically but are not solved by the existing theory of MAB processes.

While the mathematical form of Gittins indices has obtained in terms of  an essential supremum of the reward rate over a class of stopping times, it is generally difficult to precisely or numerically compute the Gittins indices, even with no restriction on switches. The only exceptions are the cases of MAB in discrete time and semi-Markovian fashion with finite states by means of achievable region method, and MAB in continuous time where every arm is a standard job, namely, it has a processing time and  a constant reward is collect on the completion of the job (Gittins et al., 2011). Under the general framework of RMAB processes, it is challenging how the Gittins indices should be computed, even if an arm evolves in a Markovian chain or semi-Markovian fashion, both with finite number of states, or an arm acts as a standard job. This difficulty reflects the impact  of the restrictions on switches, which imposes a challenging task for the Gittins index rule developed in this paper to be applied in real world problems.

In this study, the MAB processes considered are standard in the sense that the number $d$ of arms is fixed, the arms that are inactive are frozen and contribute no reward and the switches from arm to arm cause neither operation delay nor switch cost and the arms are statistically independent. When these standard conditions are violated, the behaviors of the bandit processes and thus their optimal policies are quite different. The models extend the standard MAB processes include, for example,  restless MAB processes (Whittle, 1988 and Bertsimas and Niño-Mora, 2000) in which the inactive arms evolve also, open bandit processes (Whitle 1981, Weiss, 1988 and Wu and Zhou, 2013) in which new arms can arrive from time to time, MAB processes with switching cost/delay (Banks and Sundaram, 1994 and Jun, 2004). It has turned
 out that the optimality of the Gittins index policies for the standard RMAB processes are largely dependent on those conditions.  It is unclear how the restrictions on switches among arms could impact the policies of those nonstandard MAB processes.
\bigskip

%With the theory established in this paper, we can discuss how to optimally scheduling
%a set of jobs on a single unreliable machine if the switching form one job to another can
%only be allowed when the machine is in up-time (Cai et al. 2009, 2014), which will be reported elsewhere.

%\section*{Acknowledgement}
%\footnote
{%This research of the authors are partly supported by NSFC (grant No. 71371074) and the 111 Project (grant No. B14019).}
%%%%%%%%%%%%%%%%%%%%%%%%%%%%%%%%%%%%%%%%%%%%%

\renewcommand{\baselinestretch}{0.8}
{%\small

}
\end{document}